\documentclass[12pt,oneside,english,microtype]{gtpart}   
\usepackage{pinlabel}
\usepackage[all]{xy}
\usepackage[T1]{fontenc}
\usepackage{geometry}
\usepackage{amsthm}
\usepackage{amstext}
\usepackage{amssymb}
\usepackage{graphicx}

\makeatletter
\numberwithin{equation}{section}
\numberwithin{figure}{section}
\theoremstyle{plain}
\newtheorem{thm}{\protect\theoremname}[section]
  \theoremstyle{plain}
  \newtheorem{cor}[thm]{\protect\corollaryname}
  \theoremstyle{definition}
  \newtheorem{example}[thm]{\protect\examplename}
  \theoremstyle{definition}
  \newtheorem{defn}[thm]{\protect\definitionname}
  \theoremstyle{plain}
  \newtheorem{lem}[thm]{\protect\lemmaname}
  \theoremstyle{remark}
  \newtheorem{note}[thm]{\protect\notename}
  \theoremstyle{plain}
  \newtheorem{question}[thm]{\protect\questionname}
  \theoremstyle{plain}
  \newtheorem{prop}[thm]{\protect\propositionname}
  \theoremstyle{remark}
  \newtheorem*{claim*}{\protect\claimname}
  \theoremstyle{remark}
  \newtheorem{rem}[thm]{\protect\remarkname}
  \theoremstyle{remark}
  \newtheorem*{acknowledgement*}{\protect\acknowledgementname}

\usepackage{amsmath,amsthm,amsfonts,amssymb,amscd,flafter,pinlabel}
\usepackage[parfill]{parskip}
\usepackage[mathscr]{eucal}
\usepackage{graphics}

\usepackage{subfigure}
\usepackage{graphicx}
\usepackage{caption}
\usepackage{multirow}

\def\mylist#1 {\ifx!#1\else\makebox[4em][r]{#1} \expandafter\mylist\fi}

\makeatother

\usepackage{babel}
  \providecommand{\acknowledgementname}{Acknowledgement}
  \providecommand{\claimname}{Claim}
  \providecommand{\corollaryname}{Corollary}
  \providecommand{\definitionname}{Definition}
  \providecommand{\examplename}{Example}
  \providecommand{\lemmaname}{Lemma}
  \providecommand{\notename}{Note}
  \providecommand{\propositionname}{Proposition}
  \providecommand{\questionname}{Question}
  \providecommand{\remarkname}{Remark}
\providecommand{\theoremname}{Theorem}

\author{Kenan \.{I}nce} 
\givenname{Kenan}
\surname{\.{I}nce}
\email{kince@westminsteru.edu}
\address{Department of Mathematics\\\newline
Westminster University\\\newline
1840 South 1300 East\\\newline
Salt Lake City, UT 84105\\\newline
United States}
\urladdr{http://cs.westminsteru.edu/~kince}


\arxivreference{}  
\arxivpassword{}   

\begin{document}

\title{Untwisting information from Heegaard Floer homology}

\begin{abstract}
The unknotting number of a knot is the minimum number of crossings
one must change to turn that knot into the unknot. We work with a
generalization of unknotting number due to Mathieu-Domergue, which
we call the untwisting number. The $p$-untwisting number\emph{ }is
the minimum number (over all diagrams of a knot) of full twists on
at most $2p$ strands of a knot, with half of the strands oriented
in each direction, necessary to transform that knot into the unknot.
In previous work, we showed that the unknotting and untwisting numbers
can be arbitrarily different. In this paper, we show that a common
route for obstructing low unknotting number, the Montesinos trick,
does not generalize to the untwisting number. However, we use a different
approach to get conditions on the Heegaard Floer correction terms
of the branched double cover of a knot with untwisting number one.
This allows us to obstruct several $10$ and $11$-crossing knots from being
unknotted by a single positive or negative twist. We also use the
Ozsv\'{a}th-Szab\'{o} tau invariant and the Rasmussen $s$ invariant to differentiate
between the $p$- and $q$-untwisting numbers for certain $p,q>1$.
\end{abstract}
\maketitle

\section{Introduction}

It is a natural knot-theoretic question to seek to measure ``how
knotted up'' a knot is. One such ``knottiness'' measure is given
by the \emph{unknotting number} \textbf{$u(K)$}, the minimum number
of crossings, taken over all diagrams of $K$, one must change to
turn $K$ into the unknot. By a \emph{crossing change }we shall mean
one of the two local moves on a knot diagram given in Figure \ref{fig:Crossing-changes}.

\begin{figure}
\def\svgwidth{0.5\columnwidth}
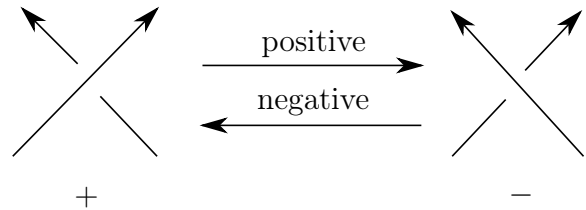

\protect\caption{\label{fig:Crossing-changes}A positive and negative crossing change.}
\end{figure}

This invariant is quite simple to define but has proven itself very
difficult to master. Fifty years ago, Milnor conjectured that the
unknotting number for the $(p,q)$-torus knot was $(p-1)(q-1)/2$;
only in 1993, in two celebrated papers \cite{kronheimer_gauge_1993,kronheimer_gauge_1995},
did Kronheimer and Mrowka prove this conjecture true. Hence, it is
desirable to look at variants of unknotting number which may be more
tractable. One natural variant (due to Murakami \cite{murakami_algebraic_1990})
is the \emph{algebraic unknotting number} $u_{a}(K)$, the minimum
number of crossing changes necessary to turn a given knot into an
Alexander polynomial-one knot. Alexander polynomial-one knots are
significant because they ``look like the unknot'' to \emph{classical
invariants}, knot invariants derived from the Seifert matrix. It is
obvious that $u_{a}(K)\leq u(K)$ for any knot $K$, and there exist
knots such that $u_{a}(K)<u(K)$ (for instance, the Whitehead double of any nontrivial knot).

In \cite{mathieu_chirurgies_1988-1}, Mathieu and Domergue defined
another generalization of unknotting number. In \cite{livingston_slicing_2002},
Livingston worked with this definition. He described it as follows: 
\begin{quotation}
``One can think of performing a crossing change as grabbing two parallel
strands of a knot with opposite orientation and giving them one full
twist. More generally, one can grab $2k$ parallel strands of $K$
with $k$ of the strands oriented in each direction and give them
one full twist.''
\end{quotation}
Following Livingston, we call such a twist a \emph{generalized crossing
change}. We describe in \cite{ince_untwisting_2015} how a crossing
change may be encoded as a $\pm1$-surgery on a nullhomologous unknot
$U\subset S^{3}-K$ bounding a disk $D$ such that $D\cap K=2$ points.
From this perspective, a generalized crossing change is a relaxing
of the previous definition to allow $D\cap K=2k$ points for any $k$,
provided $\text{lk}(K,U)=0$ (see Fig.\ \ref{fig:Generalized-crossing-change}).
In particular, any knot can be unknotted by a finite sequence of generalized
crossing changes.

\begin{figure}[h]
\def\svgwidth{0.75\columnwidth}
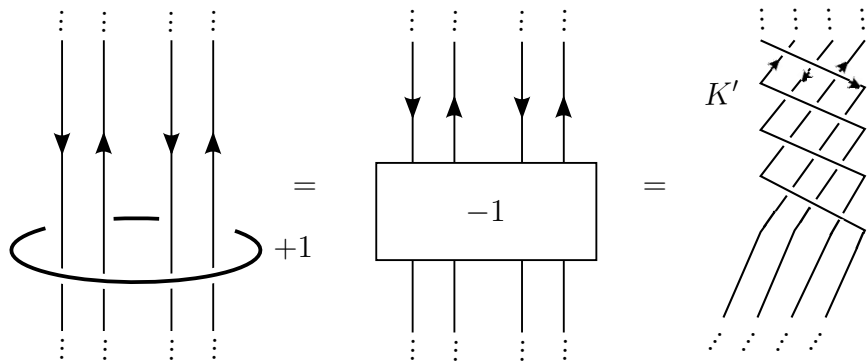

\protect\caption{\label{fig:Generalized-crossing-change}A left-handed, or positive,
generalized crossing change.}
\end{figure}

One may then naturally define the \emph{untwisting number} $tu(K)$
to be the minimum length, taken over all diagrams of $K$, of a sequence
of generalized crossing changes beginning at $K$ and resulting in
the unknot. By $tu_{p}(K)$, we will denote the minimum number of
generalized crossing changes on $2p$ or fewer strands, with $p$
strands oriented in each direction, needed to unknot $K$. Notice
that $tu_{1}=u$ and that 
\[
tu\leq\dots\leq tu_{p+1}\leq tu_{p}\leq\dots\leq tu_{1}=u.
\]

The \emph{algebraic untwisting number} $tu_{a}(K)$ is the minimum
number of generalized crossing changes, taken over all diagrams of
$K$, needed to transform $K$ into an Alexander polynomial-one knot.
It is clear that $tu_{a}(K)\leq tu(K)$ for all knots $K$. In \cite{ince_untwisting_2015},
we showed that, in fact, $tu_{a}(K)=u_{a}(K)$ for all knots $K$,
hence the unknotting and untwisting numbers are ``algebraically the
same''. However, we also showed that $tu$ and $u$ can be arbitrarily
different in general: there exists a family of knots $\{S_{p}^{q}\}$
such that $(u-tu_{q})(S_{p}^{q})\geq p-1$ for all $p,q\geq2$. 

Since the family $\{S_{p}^{q}\}$ consists of $(p,1)$-cables of (untwisted)
Whitehead doubles, most members of this family have very high crossing
number. In this paper, we compare the unknotting and untwisting numbers
for several $10$ and $11$-crossing knots with signature $0$. In order to
do this, we will develop an obstruction to a knot with signature $0$
having untwisting number $1$. This will require the methods of Heegaard
Floer homology, specifically the $d$\emph{-invariants} or \emph{Heegaard
Floer correction terms} of a $3$-manifold. 

In \cite{ozsvath_knots_2005}, Ozsv\'{a}th and Szab\'{o} develop an unknotting number $1$ obstruction using $d$-invariants. This obstruction relies on the \emph{Montesinos trick}, which allows them to construct a definite $4$-manifold with boundary the branched double cover $\Sigma(K)$
of an unknotting number-$1$ knot $K$. In Section \ref{sec:Failure-of-the-Montesinos-trick}, we give an infinite family of knots which have untwisting number $1$ but which do not satisfy the Montesinos trick, eliminating that route toward a $d$-invariant obstruction:
\begin{thm}
\label{thm:Montesinos-counterex-1}There exists an infinite family
$\{K_{n}\}_{n>1}$ of knots such that $tu(K_{n})=1$ for all $n$,
but $\Sigma(K_{n})$ is not half-integer surgery on any knot in $S^{3}$
for any $n$.
\end{thm}
In Section \ref{sec:Heegaard-Floer-theoretic-obstructions}, we get
around the failure of the Montesinos trick for untwisting number-$1$
knots by porting the machinery used by Owens and Strle in \cite{owens_immersed_2013} and Nagel and Owens in \cite{nagel_unlinking_2015}
as an obstruction to low untwisting number:
\begin{thm}
Let $K$ be a knot with signature $\sigma(K)$ which can be unknotted by $p$ positive and $n$ negative generalized crossing changes. Then
$Y=\Sigma(K)$, the branched double cover of $K$, bounds a smooth $4$-manifold $W$ with $b_{2}(W)=2n+2p$ and signature $2n-2p+\sigma(K)$.
Moreover, $H_{2}(W;\mathbb{Z})$ contains $n$ classes of self-intersection
$+2$ and $p$ classes of self-intersection $-2$ which span a primitive sublattice; in other words, the quotient of $H_{2}(W;\mathbb{Z})$ by this sublattice is torsion-free.
\end{thm}
Once we have constructed a $4$-manifold $W$ with $\partial W=\Sigma(K)$,
the next step is to apply a result of Ozsv\'{a}th-Szab\'{o} to get conditions
that the $d$-invariants of $\Sigma(K)$ must satisfy. These invariants
are easily computable for alternating $K$ via the\emph{ Goeritz matrix}
associated to $K$. These computations are discussed further in Section
\ref{sec:Heegaard-Floer-theoretic-obstructions}. We successfully
obstruct several $10$-crossing knots from being unknotted by a single
positive and/or negative generalized crossing change, though these
untwisting numbers cannot be computed using the methods available
prior to the development of Heegaard Floer homology:
\begin{thm}
\label{thm:10-crossing-untwisting-numbers}The knots $10_{68}$ and $10_{96}$ have
untwisting number $2$; the knots $10_{22}$, $10_{34}$, $10_{35}$, $10_{87}$, and $10_{90}$ cannot be unknotted by a single positive generalized crossing change; and the knot $10_{48}$ cannot be unknotted by a single negative generalized crossing change. 
\end{thm}
Similarly, we apply these obstructions to all $11$-crossing knots
with signature $0$, algebraic unknotting number $1$, and unknotting
number $2$ to get the following:
\begin{thm}
\label{thm:tu-for-11-crossing-knots}The knots $11a_{37},11a_{103},11a_{169}$,
$11a_{214}$, and $11a_{278}$ have untwisting number $2$.
\end{thm}
Finally, we showed in \cite{ince_untwisting_2015} that there can
be arbitrarily large gaps between the $p$-untwisting number and the
$1$-untwisting number (which by definition equals the unknotting
number) for several families of knots. However, we had not yet been
able to distinguish between $tu_{p}$ and $tu_{q}$ for $p,q>1$.

In Section \ref{sec:obstructions-to-tu_p}, we use invariants coming
from Heegaard Floer homology (the Ozsv\'{a}th-Szab\'{o} $\tau$ invariant)
and Khovanov homology (the Rasmussen $s$ invariant) to give lower
bounds on the $p$-untwisting number for arbitrary $p$ via the following theorem. While visiting Mark Powell at the Max Planck Institute, he suggested this theorem and outlined a proof similar to the proof of Powell and coauthors T. Cochran, S. Harvey, and A. Ray that the $\tau$ and $s$ invariants give lower bounds for their bipolar metrics (to appear in a future paper). An anonymous referee suggested a simpler approach involving the $4$-genus, detailed in section \ref{sec:obstructions-to-tu_p}.

\begin{thm}
\label{thm:tu_p-obstruction}Let $K$ be a knot which can be converted
to the unknot via $n$ generalized crossing changes, where for every
$i$, the $i$th generalized crossing change is performed on $2p_{i}$
strands. Then 
\[
|\tau(K)|\leq\sum_{i=1}^{n}p_{i}^{2}
\]
and
\[
\frac{|s(K)|}{2}\leq\sum_{i=1}^{n}p_{i}^{2}.
\]

\end{thm}
This allows us to show that there exist $p,q>1$ so that the difference
between the $p$- and $q$-untwisting numbers of several families
of knots can be made arbitrarily large:
\begin{example}
Let $K_{p^{3}}$ denote the $(p^{3},1$)-cable of a knot $K$ with
genus $1$ and $u(K)=1=\tau(K)$ (one example of such a $K$ is the
right-handed trefoil knot). We know from \cite[Section 5]{ince_untwisting_2015}
that $tu_{p^{3}}(K_{p^{3}})=1$. We may use Theorem \ref{thm:tu_p-obstruction}
to show that 
\[
tu_{p}(K_{p^{3}})-tu_{p^{3}}(K_{p^{3}})\xrightarrow{p\to\infty}\infty.
\]

\end{example}
\textbf{Convention}.\textbf{ }In this paper, all manifolds are assumed
to be smooth, compact, orientable, and connected, and all surfaces
in manifolds are assumed to be smoothly embedded. When homology groups
are given without specifying coefficients, they are assumed to have
coefficients in $\mathbb{Z}$.

\begin{acknowledgement*}
Thanks to Stefan Friedl, Maciej Borodzik, Peter Horn, Matthias Nagel,
and Mark Powell for many enlightening conversations. Thanks also to
Stefan Friedl, Matthias Nagel, Brendan Owens, and an anonymous referee for providing comments on this paper.
A Maple program written by Brendan Owens and Sa\v{s}o Strle has been very
useful in the computations in Section \ref{sec:Examples}. I would
also like to acknowledge the results of Brendan Owens and Sa\v{s}o Strle in \cite{owens_immersed_2013}, Matthias Nagel and Brendan Owens in \cite{nagel_unlinking_2015}, Brendan Owens in
\cite{owens_unknotting_2005}, and Tim Cochran and William
Lickorish in \cite{cochran_unknotting_1986}, all of which greatly
inspired this work.
\end{acknowledgement*}

\section{Preliminaries}

\subsection{\label{sub:Knot-surgery}Dehn surgery}

In this section, we will describe the operation of Dehn surgery on
knots.
\begin{defn}
\label{def:Knot-surgery}Let $K\subset S^{3}$ be an oriented knot,
let $N$ be a closed tubular neighborhood of $K$, and consider the
preferred framing for $N$ (see \cite[Definition 2E8]{rolfsen_knots_1976})
in which the longitude $L$ is oriented in the same way as $K$
and the meridian $M$ has linking number $+1$ with $K$. We may
write any simple closed curve $J\subset\partial N$ in terms of the
homology basis $\{\lambda=[L],\mu=[M]\}$: 
\[
[J]=q\lambda+p\mu\in H_{1}(\partial N).
\]

The \emph{result of $\frac{p}{q}$-surgery on $K$ }is the $3$-manifold
\[
S_{p/q}^{3}(K):=(S^{3}-\mathring{N})\bigcup_{h} (S^{1}\times D^{2}),
\]
where $h\colon\partial(S^{1}\times D^{2})\to\partial N$ is a homeomorphism taking
$\ast\times S^{1}$ onto a curve $J$ of class $[J]=p\mu+q\lambda$ in $H_{1}(\partial N)$. By convention, we indicate that surgery is to be performed on $K$
by writing the ratio $\frac{p}{q}$ next to a diagram of $K$.

If $U\subset S^{3}\setminus K$ is an unknot such that $\text{lk}(K,U)=0$,
we define a \emph{generalized crossing change diagram for $K$} to
be a diagram of the link $K\sqcup U$ with the number $\pm1$ written
next to $U$, indicating that $U$ is to have $\pm1$-surgery performed
on it.
\end{defn}
There is an orientation-preserving homeomorphism $\Phi$ of the manifold $M:=S_{\pm1}^{3}(U)$
resulting from $\pm1$-surgery on $U$ with $S^{3}$. However, $K^\prime := \Phi(K)\subset S^{3}$
may have a different knot type than $K$. (Note that the knot type
of $K^\prime$ does not depend on the choice of homeomorphism $\Phi$
since any two orientation-preserving homeomorphisms of $S^{3}$ are
isotopic.) In particular, if $D$ is a disk bounded by $U$ such that $2p$ strands of $K$ pass through $D$ in straight segments, then each of the $2p$
straight pieces is replaced by a helix which screws through a neighborhood
of $D$ in the right- (respectively, left-) hand sense (see Fig. \ref{fig:Rolfsen-twist}).

\begin{figure}
\def\svgwidth{0.5\columnwidth}
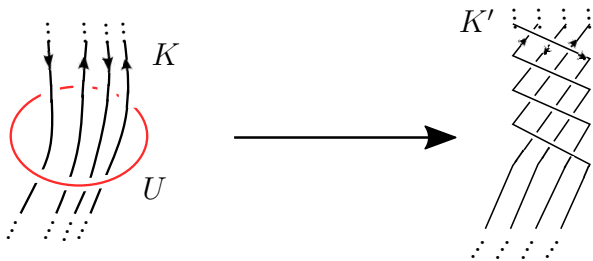

\protect\caption{\label{fig:Rolfsen-twist}Performing $+1$-surgery on an unknot
$U$ gives the knot $K$ a left-handed twist.}
\end{figure}

The process of performing $\pm 1$-surgery on an unknot $U$ in a generalized crossing change diagram for a knot $K$, mapping the resulting manifold to $S^{3}$ via an orientation-preserving homeomorphism $\Phi$, then erasing $\Phi(U)$ from the resulting diagram of $\Phi(K)\sqcup \Phi(U)$ is called a \emph{$\pm$-generalized crossing change} on $K$. Now, it can be easily verified that performing a $-$-generalized crossing change on the knot $K$ on the left side of Figure \ref{fig:Crossing-changes-as-blowdowns}
transforms the crossing labeled $+$ into the crossing labeled $-$.
The inverse process of introducing an unknot labeled with a $+1$ to the right side of
Figure \ref{fig:Crossing-changes-as-blowdowns} and performing a $+$-generalized crossing change
in the resulting generalized crossing change diagram transforms the crossing labeled $-$ into the crossing labeled $+$.

\begin{figure}
\def\svgwidth{0.5\columnwidth}
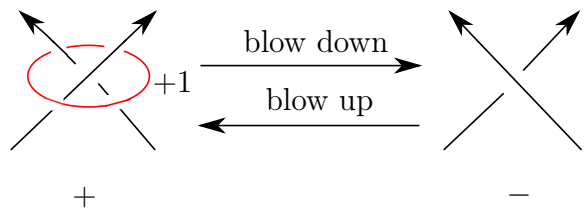

\protect\caption{\label{fig:Crossing-changes-as-blowdowns}A crossing change is a $1$-generalized crossing change.}
\end{figure}

\subsection{The untwisting number}

In a generalized crossing change diagram for $K$ consisting of a diagram
of $K$ and an unknot $U$, $K$ must pass through $U$ an even number of times, for otherwise
$\text{lk}(K,U)\neq0$. If at most $2p$ strands of $K$ pass through
an unknot $U$ in a generalized crossing change diagram, we may call the associated
$\pm$-generalized crossing change a $\pm p$\emph{-generalized crossing
change on $K$.} 

The \emph{untwisting number }$tu(K)$ \emph{of }$K$ is the minimum
length of a sequence of generalized crossing changes on $K$ such
that the result of the sequence is the unknot, where we allow ambient
isotopy of the diagram in between generalized crossing changes. Note
that by the reasoning on page 58 of \cite{adams_knot_1994}, this
definition is equivalent to taking the minimum length, over all diagrams
of $K$, of a sequence of generalized crossing changes beginning with
a fixed diagram of $K$ such that the result of the sequence is the
unknot, where we do not allow ambient isotopy of the diagram in between
generalized crossing changes. 

For $p=1,2,3,\dots$, we define the \textbf{$p$}\emph{-untwisting
number }$tu_{p}(K)$ to be the minimum length of a sequence of $\pm p$-generalized
crossing changes on $K$ resulting in the unknot, where we allow ambient
isotopy of the diagram in between generalized crossing changes. It follows immediately that we have the chain of inequalities
\begin{equation}
tu(K)\leq\dots\leq tu_{p+1}(K)\leq tu_{p}(K)\leq\dots\leq tu_{2}(K)\leq tu_{1}(K)=u(K).\label{eq:gu_p-chain}
\end{equation}

\subsection{Heegaard Floer homology}

In this section, we will recall some properties of Heegaard
Floer homology, a set of invariants of $3$-manifolds defined by
Ozsv\'{a}th and Szab\'{o}. For details, refer to their papers, in particular
\cite{ozsvath_absolutely_2003,ozsvath_floer_2003,ozsvath_knots_2005}. 	
	
Let $Y$ be an oriented rational homology $3$-sphere. Recall that one can
associate to $Y$ a set $\text{Spin}^{c}(Y)$ of \emph{$\text{spin}^{c}$ structures on $Y$}. In the case where $|H^{2}(Y;\mathbb{Z})|$
is odd, there is a canonical bijection $H^{2}(Y;\mathbb{Z})\leftrightarrow\text{Spin}^{c}(Y)$
under which $0\in H^{2}(Y;\mathbb{Z})$ is sent to the unique spin
structure on $Y$. In this way, we may give $\text{Spin}^{c}(Y)$
a group structure inherited from that of $H^{2}(Y;\mathbb{Z})$.

Fix a $\text{spin}^{c}$ structure $\mathfrak{s}$ on $Y$. Then the
\emph{(plus flavor of) Heegaard Floer homology} $HF^{+}(Y,\mathfrak{s})$
is a $\mathbb{Q}$-graded abelian group with a $\mathbb{Z}[U]$-action,
where multiplication by $U$ lowers the grading by $2$. Associated
to $\mathfrak{s}$ is a \emph{d-invariant }$d(Y,\mathfrak{s})\in\mathbb{Q}$
which satisfies the symmetry condition $d(Y,\mathfrak{s})=-d(-Y,\mathfrak{s})$.
The correction terms are useful for obstructing the existence of a
$4$-manifold with boundary $Y$:
\begin{thm}[Ozsv\'{a}th-Szab\'{o} 2003, \cite{ozsvath_absolutely_2003}]
\label{thm:O-Sz-QHS3} Let $X$ be a negative definite $4$-manifold
with boundary $Y$ and intersection form represented by a matrix $Q$,
and let $\mathfrak{s}$ be any $\text{spin}^{c}$ structure on $X$.
Let $c_{1}(\mathfrak{s})$ denote the first Chern class associated
to $\mathfrak{s}$. Then
\begin{eqnarray}
\frac{c_{1}(\mathfrak{s})^{2}+b_{2}(X)}{4} & \leq & d(Y,\mathfrak{s}|_{Y}),\nonumber \\
\text{and }\frac{c_{1}(\mathfrak{s})^{2}+b_{2}(X)}{4} & \equiv & d(Y,\mathfrak{s}|_{Y})\mod2.\label{eq:definite-sharp-conditions}
\end{eqnarray}

\end{thm}
Following \cite{ozsvath_absolutely_2003}, we now show how to check this obstruction in practice. In addition to the assumptions of Theorem \ref{thm:O-Sz-QHS3}, suppose for simplicity
that $\pi_{1}(X)=1$ and that $|H^{2}(Y;\mathbb{Z})|$ is odd. (This
will always be true for the examples in this paper.) Let $r=b_{2}(X)$,
the second Betti number of $X$. It is straightforward to see that
$H_{2}(X;\mathbb{Z})$ is free of rank $r$ in this case. Choose a basis $\{x_{i}\}_{i=1}^{r}$ for $H_{2}(X;\mathbb{Z})$ and let $Q=(Q_{ij})$ be a negative definite $r\times r$ matrix representing the intersection pairing
of $X$ in this basis; then $\det Q=|H^{2}(Y;\mathbb{Z})|$. The
dual basis $\{x^{i}\}_{i=1}^{r}$ for $H^{2}(X;\mathbb{Z})$ given
by the Universal Coefficient Theorem defines an isomorphism $H^{2}(X;\mathbb{Z})\cong\mathbb{Z}^{r}$.
Under this isomorphism, the set $\{c_{1}(\mathfrak{s})|\mathfrak{s}\in\text{Spin}^{c}(X)\}\subset H^{2}(X;\mathbb{Z})$ of first Chern classes of $\text{spin}^{c}$ structures on $X$ is sent to the set of characteristic covectors $\text{Char}(Q)$ for
$Q$. (Recall that a \emph{characteristic covector} for an $r\times r$
matrix $Q$ is a vector $\xi=(\xi_{1},\dots,\xi_{r})\in\mathbb{Z}^{r}$
such that $\xi_{i}\equiv Q_{ii}\mod2$ for $i=1,\dots,r$.) In our
basis, the square of the first Chern class of the $\text{spin}^{c}$
structure corresponding to a characteristic covector $\xi$ is given
by $\xi^{T}Q^{-1}\xi$.

Define a function $m_{Q}:\mathbb{Z}^{r}/Q(\mathbb{Z}^{r})\to\mathbb{Q}$
by 
\[
m_{Q}(g)=\max\Big\{\frac{\xi^{T}Q^{-1}\xi+r}{4}\Big|\xi\in\text{Char}(Q),[\xi]=g\Big\}
\]
where $[\xi]$ is the image of $\xi\in\mathbb{Z}^{r}$ under the projection
to $\mathbb{Z}^{r}/Q(\mathbb{Z}^{r})$. In computing $m_{Q}$, it
is enough to consider characteristic covectors $\xi=(\xi_{1},\dots,\xi_{r})$
with $-Q_{ii}\geq\xi_{i}\geq Q_{ii}$; if, say, $\xi_{i}<Q_{ii}$,
subtracting twice the $i$th column of $Q$ from $\xi$ shows that
$\xi^{T}Q^{-1}\xi$ is not maximal. Then we may simplify the conditions
(\ref{eq:definite-sharp-conditions}) as follows:
\begin{thm}[Ozsv\'{a}th-Szab\'{o}]
\label{thm:m_Q-theorem} Let $Y$ be a rational homology $3$-sphere
which is the boundary of a simply connected, negative definite $4$-manifold
$X$, with $|H^{2}(Y;\mathbb{Z})|$ odd. If the intersection pairing
of $X$ is represented in a basis by the matrix $Q$, then there is
a group isomorphism
\[
\phi:\mathbb{Z}^{r}/Q(\mathbb{Z}^{r})\to\operatorname{Spin}^{c}(Y)
\]
such that 
\begin{eqnarray}
m_{Q}(g) & \leq & d(Y,\phi(g))\label{eq:m_Q-inequality}\\
m_{Q}(g) & \equiv & d(Y,\phi(g))\mod2\nonumber 
\end{eqnarray}
for all $g\in\mathbb{Z}^{r}/Q(\mathbb{Z}^{r})$.
\end{thm}
Under the assumptions of the theorem, we say that the $4$-manifold
$X$ bounded by $Y$ is \emph{sharp} if equality holds in (\ref{eq:m_Q-inequality}).
In this case, we may compute the correction terms for $Y$ using the
values of $m_{Q}$. Moreover, if a rational homology sphere $Y$ bounds
a \emph{positive }definite $4$-manifold $X$, we may compute the
correction terms for $Y$ by applying Theorem \ref{thm:m_Q-theorem}
to $-Y$.

If $K$ is an alternating knot, we may compute the $d$-invariants
of $\Sigma(K)$ using the negative definite \emph{Goeritz matrix}
computed from an alternating diagram of $K$ as follows. Consider
a regular projection of $K$ into a plane $\mathbb{R}^{2}\subset\mathbb{R}^{3}=S^{3}\setminus\{\infty\}$.
Color the regions of $\mathbb{R}^{2}\setminus K$ alternately black
and white so that the $n$ white regions $X_{1},\dots,X_{n}$ are
separated by crossings of the type depicted in the figure below.

\begin{figure}[h]
\includegraphics{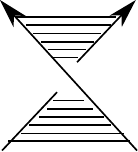}

\protect\caption{Crossing conventions for negative definite Goeritz matrices of alternating
knots.}
\end{figure}

For $0\leq i,j\leq n$, define
\[
g_{ij}=\begin{cases}
d,\text{ where \ensuremath{d} is the number of double points incident to \ensuremath{X_{i}} and \ensuremath{X_{j}},\text{ }} & i\neq j\\
-\sum_{k\neq i}g_{ik}, & i=j.
\end{cases}.
\]
Let $G^{\prime}=(g_{ij})$. Then the \emph{negative definite Goeritz
matrix} $G$ associated to $K$ is the $n\times n$ symmetric integer
matrix obtained from $G^{\prime}$ by deleting the $0$th row and
column of $G^{\prime}$. It is shown in \cite[Proposition 3.2]{ozsvath_knots_2005} that $G$ represents the intersection pairing of a sharp $4$-manifold with boundary $\Sigma(K)$; thus, the correction terms for $\Sigma(K)$ are given by the values of $m_{G}$.

\section{\label{sec:Failure-of-the-Montesinos-trick}Failure of the Montesinos
trick}

The ``Montesinos trick'' relates crossing changes downstairs on
$K$ to surgery upstairs on $\Sigma(K)$, the branched double cover
of $K$. We use the convention that the determinant of a knot is given
by $|\Delta_{K}(-1)|$, where $\Delta_{K}$ is the Alexander polynomial
for the knot $K$.
\begin{thm}
\cite{montesinos_three_1976}\label{thm:Montesinos-theorem} If $u(K)=1$,
then $\Sigma(K)\cong S_{\pm D/2}^{3}(C)$ for some other knot $C\subset S^{3}$,
where here $D$ is the determinant of $K$. 
\end{thm}
We show that this theorem does not generalize to untwisting number-$1$
knots:
\begin{thm}
\label{thm:Montesinos-counterex}There exists an infinite family $\{K_{n}\}$
of knots such that, for all $n\geq1$, $tu(K_{n})=1$, but $\Sigma(K_{n})$
is not half-integer surgery on any knot in $S^{3}$.
\end{thm}
In order to prove Theorem \ref{thm:Montesinos-counterex}, we will
need two main ingredients. The first is the following lemma:
\begin{lem}
The effect of performing a $+$-generalized crossing change on the
unknot $U$ in the local picture given in Figure \ref{fig:double-grab-gcc}
is to add $-4$ full twists to the knot $K$.
\end{lem}
\begin{figure}
\def\svgwidth{0.5\columnwidth}
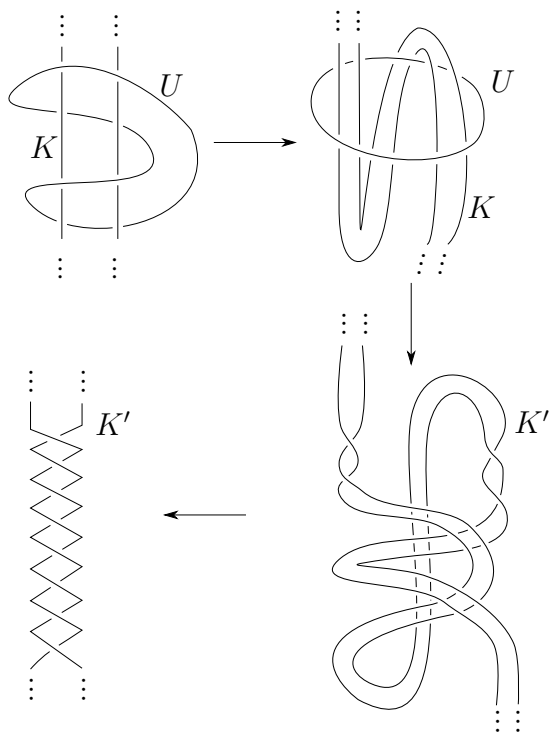

\protect\caption{\label{fig:double-grab-gcc}The (local) effect of performing a $+$-generalized
crossing change on the unknot $U$.}
\end{figure}

\begin{proof}
See Figure \ref{fig:double-grab-gcc}. The intermediate steps are
left to the reader.
\end{proof}
Our second ingredient is the following theorem of McCoy \cite{mccoy_alternating_2013}:
\begin{thm}
\label{thm:McCoy-main-thm}Let $K$ be an alternating knot. Then the
following are equivalent:
\begin{enumerate}
\item $u(K)=1$;
\item the branched double cover $\Sigma(K)$ can be obtained by half-integer
surgery on some knot in $S^{3}$;
\item in every minimal diagram of $K$, there exists a crossing which can
be changed to unknot that diagram.
\end{enumerate}
\end{thm}
\begin{proof}[Proof of Theorem \ref{thm:Montesinos-counterex}]
 Fix an orientation on $K_{n}$. The generalized crossing change
pictured in Figure \ref{fig:Montesinos-counterexample-family} introduces
$-4$ twists on the left side of $K_{n}$, which undo the $4$ twists
already present. Hence, $tu(K_{n})=1$ for all $n$. Moreover, if
$n>1$, then $K_{n}$ is a minimal diagram of an alternating knot.
One can easily see that $K_{n}$ cannot be unknotted by any single crossing
change in this diagram. By Theorem \ref{thm:McCoy-main-thm}, the
branched double cover $\Sigma(K_{n})$ cannot be obtained by half-integer
surgery on any knot in $S^{3}$, and moreover $u(K_{n})>1$.
\end{proof}
\begin{figure}
\def\svgwidth{0.5\columnwidth}
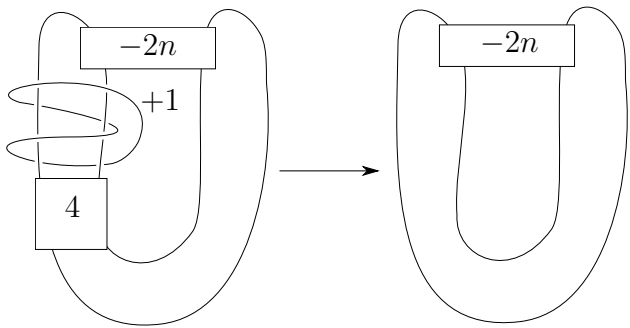

\protect\caption{\label{fig:Montesinos-counterexample-family}The knots $K_{n}$, together
with the $+1$-generalized crossing change that unknots them. Here,
positive (resp., negative) numbers in boxes denote right-handed (resp.,
left-handed) full twists.}
\end{figure}

\begin{note}
The first knot in this family is $K_{2}=12a_{1166}$. The unknotting
number of $12a_{1166}$ is listed as ``not known'' in the KnotInfo
tables, but is either $1$ or $2$. By Theorem \ref{thm:McCoy-main-thm},
we must have that $tu(12a_{1166})=1<2=u(12a_{1166})$.\end{note}
\begin{question}
Does there exist a knot $K$ with $tu(K)=1$ such that $\Sigma(K)$
is not surgery on \emph{any }knot in $S^{3}$?
\end{question}

\section{\label{sec:Heegaard-Floer-theoretic-obstructions}Heegaard Floer-theoretic
obstructions to untwisting number $1$}

Although the Montesinos trick does not hold for knots with untwisting
number $1$, we can still get obstructions to a knot $K$ being unknotted
by a single positive or negative generalized crossing change using
techniques similar to those of Owens and Strle in \cite{owens_immersed_2013} and Nagel and Owens in \cite{nagel_unlinking_2015} together with Theorem \ref{thm:O-Sz-QHS3}.

In order to apply Theorem \ref{thm:O-Sz-QHS3}, we first compute a
Goeritz matrix $G$ for $K$ and, from $G$, the function $m_{G}$
as in Proposition \ref{thm:O-Sz-QHS3}. The image of $\mathbb{Z}^{r}/G(\mathbb{Z}^{r})$
under $m_{G}$, where $G$ is an $r\times r$ matrix, is the set of
$d$-invariants for $Y$. We construct the $4$-manifold $W$ as in
\cite[Proposition 2.3]{nagel_unlinking_2015} using the propositions
below, then compute the $m_{Q}$ and show that no isomorphism satisfying
both of the conditions (\ref{eq:definite-sharp-conditions}) exists. 
\begin{prop}
\label{prop:disk-from-untwisting-sequence}Let $K$ be an oriented
knot in $S^{3}$, and suppose that $K$ can be unknotted by $p$ positive and $n$ negative generalized crossing changes. Then $K$ bounds a disk $\Delta$ in
a manifold $C\cong B^{4}\#_{n}\mathbb{CP}^{2}\#_{p}\overline{\mathbb{CP}^{2}}$ with $[\Delta]=0\in H_{2}(C,\partial C)$ and $\pi_{1}(C\setminus\Delta)=\mathbb{Z}$,
generated by a meridian of $K$.\end{prop}
\begin{proof}
Suppose that $K$ is an oriented knot in $S^{3}$ and that $K$ can
be unknotted by $p$ positive and $n$ negative generalized crossing changes. Then there is a sequence of knots
\begin{equation}
K:=K_{p+n}\xrightarrow{\epsilon_{p+n}} K_{p+n-1}\xrightarrow{\epsilon_{p+n-1}}\dots\xrightarrow{\epsilon_{2}} K_{1}\xrightarrow{\epsilon_{1}} K_{0}:=U
\end{equation}
for which $K_{i}$ is obtained from $K_{i+1}$ by a single generalized crossing change of sign $\epsilon_{i+1}\in \{\pm1\}$ for $i=1,\dots,p+n$, with precisely $p$ of the $\epsilon_{i}$ equal to $+1$ and $n$ of the $\epsilon_{i}$ equal to $-1$, and $U$ is the unknot.
Reversing our point of view, there is a sequence of knots
\begin{equation}
U:=K_{0}\xrightarrow{-\epsilon_{1}} K_{1}\xrightarrow{-\epsilon_{2}}\dots\xrightarrow{-\epsilon_{p+n-1}} K_{p+n-1}\xrightarrow{-\epsilon_{p+n}} K_{p+n}=:K\label{eq:untwisting-sequence}
\end{equation}
for which $K_{i}$ is obtained from $K_{i-1}$ by a single generalized crossing change of sign $-\epsilon_{i}$ for $i=1,\dots,p+n$ and $U$ is the unknot.

Consider $U$ to be embedded in $\partial B^{4}=S^{3}$. Since $U$
is an unknot in $S^{3}$, it bounds an embedded disk $D\subset S^{3}$.
We push $D$ into $B^{4}$ to get a disk $\Delta_{0}\subset B^{4}$
such that $\Delta_{0}\cap\partial B^{4}=U$ and $\pi_{1}(B^{4}\setminus\Delta_{0})=\mathbb{Z}$,
where the latter is generated by a meridian of $U$. 

Now, we build a $4$-manifold $C$ in which $K$ bounds a disk $\Delta$
as follows. Let $C_{0}:=B^{4}$. We now build $C$ from $C_{0}$ by
sequentially thickening the boundary of $C_{0}$ and attaching $2$-handles
to the new boundary. First, we thicken the boundary $S_{0}:=\partial B^{4}$
to $S_{0}\times [0,1]$, obtaining a new $4$-manifold $B_{0}$. We denote
the disk $\Delta_{0}\cup(U\times I)\subset B_{0}$ by $\Delta_{1}$.
The first generalized crossing change can be realized via the attachment
of a $-\epsilon_{1}$-framed $2$-handle $h_{1}$ along an unknot $U_{1}\subset S_{0}\times\{1\}$
with $\text{lk}(U\times\{1\},U_{1})=0$. There is a unique orientation-preserving
diffeomorphism from the new boundary $S_{1}$ resulting from this
handle attachment to $S^{3}$, and after this diffeomorphism $U\times\{1\}$
is isotopic to $K_{1}$. We denote by $C_{1}$ the new $4$-manifold
resulting from this handle attachment. Since attaching a $\pm 1$-framed
$2$-handle to the boundary of a $4$-manifold along an unknot results
in connect-summing a $\pm\mathbb{CP}^{2}$, we have that $C_{1}\cong C_{0}\#-\epsilon_{1}\mathbb{CP}^{2}=B^{4}\#-\epsilon_{1}\mathbb{CP}^{2}$ (here $\pm \mathbb{CP}^{2}$ denotes $\mathbb{CP}^{2}$ or $\overline{\mathbb{CP}^{2}}$, respectively). Note that $\Delta_{1}$ is still a disk in $C_{1}$ and that $\partial\Delta_{1}=K_{1}$. 
 
\begin{figure}
\def\svgwidth{0.6\columnwidth}
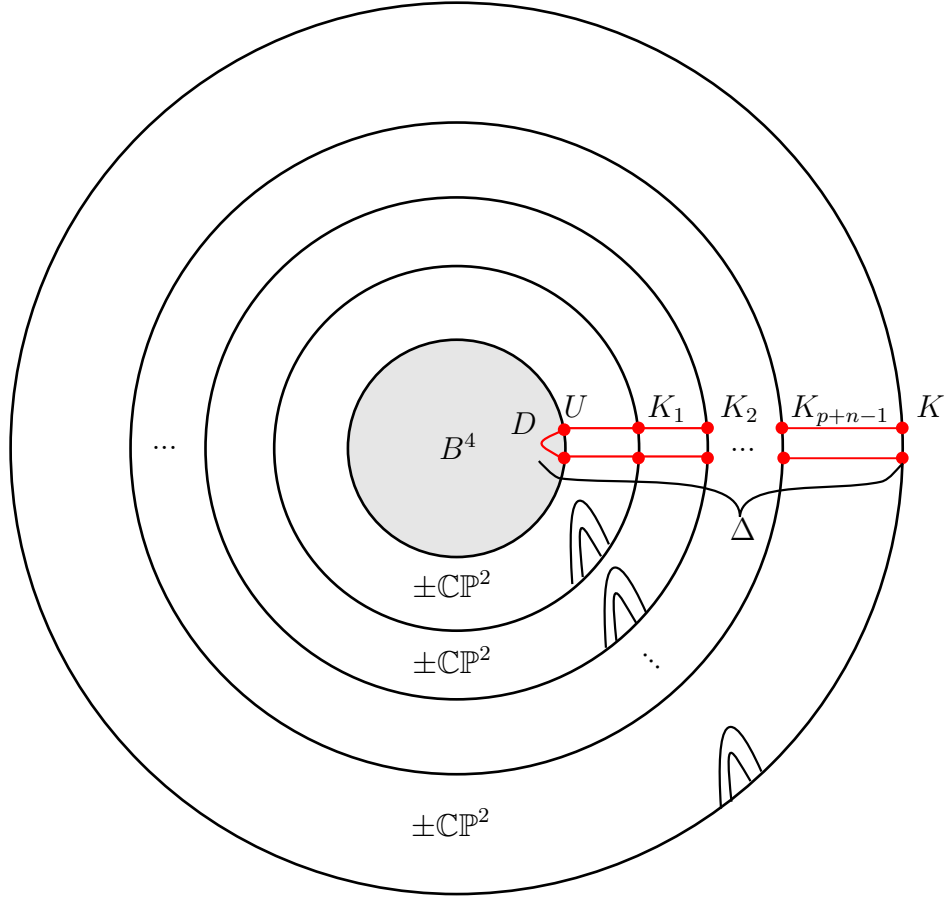

\protect\caption{\label{fig:construction-of-C}The construction of a manifold $C$
in which $K$ bounds a disk $\Delta$.}
\end{figure}

Attaching a $2$-handle generally adds a relation to the fundamental
group of the resulting manifold, where the relation is given by the
attaching map. Since the attaching circle $U_{1}$ of $h_{1}$ is
trivial in $H_{1}((S_{0}\times\{1\})\setminus(U\times\{1\}))\cong\mathbb{Z}\langle\mu_{0}\rangle$,
where $\mu_{0}$ is a meridian of $U\times\{1\}\subset S_{0}\times\{1\}$,
it is also trivial in $\pi_{1}(B_{0}\setminus\Delta_{1})\cong\mathbb{Z}\langle\mu_{0}\rangle$.
Thus, we get that $\pi_{1}(C_{1}\setminus\Delta_{1})\cong\mathbb{Z}$
as well, generated by a meridian $\mu_{1}$ of $K_{1}$.

We continue in this way to iteratively get $4$-manifolds $C_{1},\dots,C_{p+n}$
so that $C_{i+1}$ is obtained from $C_{i}$ by adding a collar $\partial C_{i}\times [i,i+1]$
to $\partial C_{i}$ and attaching a $-\epsilon_{i+1}$-framed $2$-handle $h_{i+1}$
to $\partial C_{i}\times\{i+1\}$. At each stage, the attaching circle
$U_{i+1}\subset S_{i}\times i+1$ of $h_{i+1}$ is trivial in 
\[
H_{1}((S_{i}\times\{i+1\})\setminus(K_{i}\times\{i+1\}))\cong\mathbb{Z}\langle\mu_{i}\rangle,
\]
where $\mu_{i}$ is a meridian of $K_{i}\times i+1$. Hence, $U_{i+1}$ is trivial in $\pi_{1}(B_{i}\setminus\Delta_{i+1})\cong\mathbb{Z}\langle\mu_{i}\rangle$.
The end result of this process is a $4$-manifold $C:=C_{p+n}\cong B^{4}\#_{n}\mathbb{CP}^{2} \#_{p}\overline{\mathbb{CP}^{2}}$
in which $K:=K_{p+n}$ bounds a disk $\Delta:=\Delta_{p+n}$ such that
$\pi_{1}(C\setminus\Delta)\cong\mathbb{Z}$, generated by a meridian
$\mu_{p+n}$ of $K=K_{p+n}$.

We now consider the non-degenerate intersection form $H_{2}(C,\partial C)\times H_{2}(C)\to\mathbb{Z}$ in order to show that $[\Delta]=0\in H_{2}(C,\partial C)$.
Since $H_{2}(C)\cong\mathbb{Z}^{p+n}$ is generated by the $\mathbb{CP}^{1}$
factors $\mathbb{CP}_{1}^{1},\dots,\mathbb{CP}_{p+n}^{1}$, where
$\mathbb{CP}_{i}^{1}$ is a generator of the second homology of the
$i$th connect-summed copy of $\pm\mathbb{CP}^{2}$, we know that an
element $a\in H_{2}(C,\partial C)$ is $0$ if and only if $a\cdot[\mathbb{CP}_{i}^{1}]=0$
for all $i=1,\dots,p+n$. 

Let $d_{i}$ denote the disk bounded by the unknot $U_{i}$, and let $D_{i}$ denote the second $D^{2}$ factor in the $i$th
$2$-handle attached to $C$. Then $\mathbb{CP}_{i}^{1}$ is homologous to 
\[
\Big(d_{i-1}\times\Big\{i-\frac{1}{2}\Big\}\Big)\cup\Big(U_{i}\times\Big[i-\frac{1}{2},i\Big]\Big)\cup \Big(\ast\times D_{i}\Big).
\]

The only intersections of $\Delta$ with $\mathbb{CP}_{i}^{1}$ come
from the intersections of $K_{i-1}$ with $d_{i}$. Since $\text{lk}(K_{i-1},U_{i})=0$
for all $i$, [$K_{i-1}]\cdot [d_{i}]=0$ for all $i$.
Therefore, $[\Delta]=0\in H_{2}(C,\partial C)$. This completes the
proof of the proposition.
\end{proof}
Next, we prove a generalization of \cite[Proposition 2.3]{nagel_unlinking_2015}:
\begin{prop}
\label{prop:(Nagel-Owens-Proposition-2.3)}Let $K$ be a knot in $S^{3}=\partial B^{4}$,
and suppose $K$ bounds a properly embedded disk $\Delta$ in $C:=B^{4}\#_{n}\mathbb{CP}^{2}\#_{p}\overline{\mathbb{CP}^{2}}$
such that $[\Delta]=0\in H_{2}(C,\partial C)$ and $\pi_{1}(C\setminus\Delta)=\mathbb{Z}$, generated by a
meridian of $K$. Then there exists an oriented $4$-manifold $W$
with boundary $\partial W=\Sigma(K)$, the branched double cover of $K$, such that
\begin{enumerate}
\item $W$ is simply connected;
\item $H_{2}(W;\mathbb{Z})\cong\mathbb{Z}^{2(p+n)}$;
\item the signature of $W$ is $\sigma(W)=2(n-p)+\sigma(K)$;
\item there exist $p+n$ pairwise disjoint classes in $H_{2}(W;\mathbb{Z})$
represented by $p$ surfaces of self-intersection $-2$ and $n$ surfaces of self-intersection $+2$ which span a primitive sublattice; in other words, the quotient of $H_{2}(W;\mathbb{Z})$ by this sublattice is torsion-free.
\end{enumerate}
\end{prop}
\begin{proof}
Since $\pi_{1}(C\setminus\Delta)=\mathbb{Z}$ with generator the meridian
of $K$, we may take the double cover $W=\Sigma_{2}(C,\Delta)$ of
$C$ branched along $\Delta$, and by definition we have $\partial W=\Sigma_{2}(K)$. 
\begin{claim*}
$W$ is simply connected.
\end{claim*}
Let $p\colon(\widetilde{C\setminus N(\Delta)})\to C\setminus N(\Delta)$
denote the two-fold, unbranched cover of the complement of an open
tubular neighborhood of $\Delta$ in $C$. Since $\pi_{1}(C\setminus\Delta)\cong\mathbb{Z}$,
we have that $\pi_{1}(\widetilde{C\setminus\Delta})\cong\mathbb{Z}$
as well. The branched cover $W$ may be recovered from $\widetilde{C\setminus N(\Delta)}$
by gluing back a closed neighborhood $\overline{N(\Delta)}\cong D^{2}\times\Delta$
along $p^{-1}(\partial\overline{N(\Delta)})\cong S^{1}\times\Delta$.
A straightforward application of the Seifert-van Kampen theorem to
$W=\widetilde{C\setminus\Delta}\cup_{p^{-1}(\partial\overline{N(\Delta)})}\overline{N(\Delta)}$
shows that $\pi_{1}(W)=1$.
\begin{claim*}
The Euler characteristic of $W$ is $\chi(W)=2(p+n)+1$.
\end{claim*}
By a standard Mayer-Vietoris argument, we may show that
\[
H_{i}(C)=\begin{cases}
\mathbb{Z}, & i=0\\
0, & i=1,3\\
\mathbb{Z}^{p+n}, & i=2\\
0, & i=4
\end{cases}
\]
where $H_{4}(C)=0$ because $C$ is a manifold with boundary. Thus,
$\chi(C)=1+p+n$. We have that
\[
\chi(C)=\chi(C\setminus\Delta)+\chi(\Delta)=\chi(C\setminus\Delta)+1.
\]
Therefore, the double cover $\widetilde{C\setminus\Delta}$ of $C\setminus\Delta$
has Euler characteristic $2(\chi(C)-1)$. Since $W=\widetilde{C\setminus\Delta}\cup_{p^{-1}(\partial\overline{N(\Delta)})}\overline{N(\Delta)}$
as above, we have that
\[
\chi(W)=2(\chi(C)-1)+1=2(p+n)+1.
\]

\begin{claim*}
The second homology of $W$ is $\mathbb{Z}^{2(p+n)}$.
\end{claim*}
Since $H_{1}(W;\mathbb{Z})=0$, the Universal Coefficient Theorem together with the long exact cohomology sequence for $(X,\partial X)$ implies that $H^{1}(W,\partial W;\mathbb{Z})=0$ as well. By Poincar\'{e}-Lefschetz
duality, we have that $H_{3}(W;\mathbb{Z})=0$ as well. Note that
$H_{4}(W;\mathbb{Z})=0$ since $W$ is a manifold with boundary. Now
the Euler characteristic of $W$ is
\[
2p+2n+1=\chi(W)=1+b_{2}(W).
\]
Therefore, $b_{2}(W)=2(p+n)$, and $H_{2}(W;\mathbb{Z})$ is free abelian of rank $2(p+n)$. This completes
the proof of the Claim.
\begin{claim*}
The signature of $W$ is $\sigma(W)=2(n-p)+\sigma(K)$.
\end{claim*}
Our proof follows the proof of \cite[Theorem 3.7]{cochran_unknotting_1986}.
Let $F_{-K}$ be a connected Seifert surface of the knot $-K$ with
interior pushed into $-B^{4}$. Then the manifold $(\hat{C},F):=(C,\Delta)\cup_{(S^{3},K)}(-B^{4},F_{-K})$
is closed. Let $\hat{W}$ denote the double cover $\Sigma_{2}(\hat{C},F)$
of $\hat{C}$ branched over $F:=\Delta\cup_{K}F_{-K}$. Then $\hat{W}=W\cup_{\Sigma_{2}(K)}X_{K}$,
where $X_{K}$ is the double cover $\Sigma_{2}(F_{-K})$ of $-B^{4}$
branched along $F_{-K}$. By \cite{viro_branched_1973,kauffman_signature_1976},
the signature of $X_{K}$ is $-\sigma(K)$. Applying Novikov additivity,
we get that
\[
\sigma(\hat{W})=\sigma(W)+\sigma(X_{K}).
\]
Furthermore, the $G$-signature theorem \cite[Lemma 2.1]{casson_slice_1978}
tells us that
\[
\sigma(\hat{W})=2\sigma(\hat{C})-\frac{1}{2}([F]\cdot[F]).
\]
Since in this case $[\Delta]=0\in H_{2}(C,\partial C)$, we
have that $[F]\cdot[F]=0$ so 
\[
\sigma(W)=2\sigma(C)+\sigma(K).
\]
Since $\sigma(C)=n-p$, we get that $\sigma(W)=2(n-p)+\sigma(K)$. This
completes the proof of the Claim.
\begin{claim*}
There exist $p+n$ pairwise disjoint classes in $H_{2}(W;\mathbb{Z})$,
$p$ of self-intersection $-2$ and $n$ of self-intersection $+2$, which span a primitive sublattice.
\end{claim*}
We let $S_{i}$ be a smoothly embedded surface representing the generator of $H_{2}(-\epsilon_{i}\mathbb{CP}_{i}^{2})$, the $i$th summand of $C$. We define $x_{i}$ to be the homology class
of the two-fold cover $\hat{S_{i}}\subset W$ of $S_{i}$ branched
over $\Delta\cap S_{i}$, which is a subset of $W$. Since the $S_{i}$ are pairwise
disjoint downstairs, the $\hat{S_{i}}$ are also pairwise disjoint. We show that the $x_{i}$ have self-intersection $-2\epsilon_{i}$. 

Let $S_{i}^{+}$ be a push-off of $S_{i}$. Then $S_{i} \cdot S_{i}^{+} = -\epsilon_{i}$. We make the disk $\Delta$ disjoint from the (codimension-$2$) intersection points $S_{i} \cap S_{i}^{+}$. In the branched cover, denote the preimage of $S_{i}$ by $T_{i}$ and the preimage of $S_{i}^{+}$ by $T_{i}^{+}$. Then $T_{i}^{+}$ is also a push-off of $T_{i}$. The intersection points of $T_{i}$ and $T_{i}^{+}$ are the preimages of the intersection points of $S_{i}$ and $S_{i}^{+}$; since the points of $S_{i} \cap S_{i}^{+}$ are disjoint from the branch set, there are geometrically two intersection points of $T_{i}$ and $T_{i}^{+}$. Furthermore, the orientations upstairs give the same signs of intersection as downstairs. Therefore, $T_{i} \cdot T_{i}^{+} = -2\epsilon_{i}$. 

The proof of \cite[Proposition 2.3]{nagel_unlinking_2015} applies
to our case to show that these classes span a primitive sublattice.
This completes the proof of the Proposition.\end{proof}
\begin{rem}
The proof of Proposition \ref{prop:(Nagel-Owens-Proposition-2.3)}
is very similar to the proof of \cite[Proposition 2.3]{nagel_unlinking_2015},
with the caveat that Nagel and Owens use only $-1$-generalized crossing changes in order to unknot $K$, no matter the signs of the crossings
of $K$ that need to be changed (see Figure \ref{fig:Nagel-Owens-(-1)s}).
The diagram on the right side of Figure \ref{fig:Nagel-Owens-(-1)s} is not a generalized crossing change diagram, since $\text{lk}(K,U)\neq0$. Therefore, we must assume
that $K$ can be unknotted only by positive generalized crossing changes.

\begin{figure}
\def\svgwidth{0.5\columnwidth}
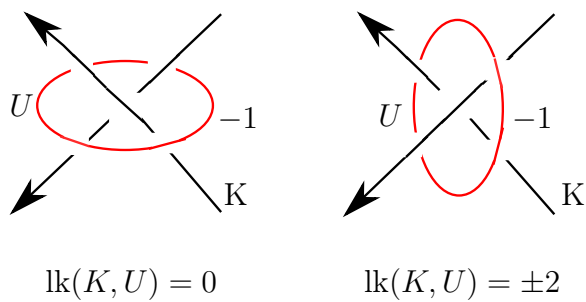

\protect\caption{\label{fig:Nagel-Owens-(-1)s}No matter the sign of the crossing to
be changed, Nagel and Owens \cite{nagel_unlinking_2015} may perform only $-1$-generalized crossing changes in order to do so.}
\end{figure}

\end{rem}
From Propositions \ref{prop:disk-from-untwisting-sequence} and \ref{prop:(Nagel-Owens-Proposition-2.3)},
we derive a theorem analogous to \cite[Theorem 1]{nagel_unlinking_2015},
but requiring the additional condition that the signature of the knot
$K$ is $0$:
\begin{thm}
\label{thm:signature-0-tu-manifold}Let $K\subset S^{3}$ be an oriented
knot with signature $0$ which can be unknotted by $p$ generalized
crossing changes, all of sign $+1$. Then the double cover $Y:=\Sigma(K)$
of $S^{3}$ branched along $K$ bounds a smooth, simply connected,
negative definite $4$-manifold $W$ with $H_{2}(W;\mathbb{Z})\cong\mathbb{Z}^{2p}$.
Moreover, $H_{2}(W;\mathbb{Z})$ contains $p$ pairwise disjoint homology
classes of self-intersection $-2$ which span a primitive sublattice.\end{thm}
\begin{proof}
By Proposition \ref{prop:disk-from-untwisting-sequence}, $K$ bounds
a disk $\Delta$ in a manifold $C\cong B^{4}\#_{p}\mathbb{CP}^{2}$
such that $[\Delta]=0\in H_{2}(C,\partial C)$ and $\pi_{1}(C\setminus\Delta)=\mathbb{Z}\langle\mu\rangle$,
where $\mu$ is a meridian of $K$. By Proposition \ref{prop:(Nagel-Owens-Proposition-2.3)},
the double cover $W:=\Sigma_{2}(C,\Delta)$ of $C$ branched over
$\Delta$ is simply connected, has $H_{2}(W;\mathbb{Z})\cong\mathbb{Z}^{2p}$,
and contains $p$ pairwise disjoint homology classes of self-intersection
$-2$ which span a primitive sublattice. Moreover, the signature of
$W$ is $\sigma(W)=-2p+\sigma(K)=-2p$, so that $W$ is negative definite. \end{proof}
\begin{note}
If instead $K$ can be unknotted using $n$ generalized crossing changes,
all of sign $-1$, Theorem \ref{thm:signature-0-tu-manifold} applied
to $-K$ shows that the double cover $-Y$ of $S^{3}$ branched along
$-K$ bounds a smooth negative definite $4$-manifold $W$ with $b_{1}(W)=0,b_{2}(W)=2n$,
and such that $H_{2}(W;\mathbb{Z})$ contains $n$ pairwise disjoint
surface classes of self-intersection $-2$ which span a primitive
sublattice.
\end{note}
In the rest of this paper, we will say $tu(K)=\pm1$ if $K$ can be
unknotted by a single $\pm$-generalized crossing change. If $\sigma(K)=0$
and $tu(K)=\pm1$, we can always get a negative definite $4$-manifold
$W$ bounding $\pm\Sigma(K)$: if $K$ can be unknotted by a positive
generalized crossing change, then we get a negative definite $W$
bounding $+\Sigma(K)$, and if $K$ can be unknotted by a negative
generalized crossing change, then we get a negative definite $W$
bounding $-\Sigma(K)$. Moreover, the intersection form on $W$ is
represented by a definite $2\times2$ matrix $Q$. For an $n\times n$
matrix $M$, we denote by $\Gamma_{M}$ the group $\mathbb{Z}^{n}/M(\mathbb{Z}^{n})$.
With this terminology established, we may state the following corollary
of Theorem \ref{prop:disk-from-untwisting-sequence}, which simplifies
our computations:
\begin{cor}
\label{cor:Owens-obstruction-for-alternating-knots}Let $K$ be an
alternating knot such that $tu(K)=\pm1$ and $\sigma(K)=0$. We use
the convention that $\det K=|\Delta_{K}(-1)|>0$. Let $G$ be the
negative definite Goeritz matrix obtained from an alternating diagram
for $\pm K$. Then there exists a negative definite matrix of the
form 
\[
Q=\begin{pmatrix}-\frac{\det K+1}{2} & 1\\
1 & -2
\end{pmatrix}
\]
so that $\pm Y=\pm\Sigma(K)$ bounds a negative definite $4$-manifold
with intersection form $Q$. Moreover, there is an isomorphism $\phi:\Gamma_{Q}\to\Gamma_{G}$
such that
\begin{eqnarray}
m_{Q}(g) & \leq & m_{G}(\phi(g))\label{eq:positive-matching}\\
m_{Q}(g) & \equiv & m_{G}(\phi(g))\mod2\label{eq:even-matching}
\end{eqnarray}
for all $g\in\Gamma_{Q}$.\end{cor}
\begin{proof}
By Theorem \ref{thm:signature-0-tu-manifold}, $\pm Y$ bounds a negative
definite $4$-manifold with intersection form represented by 
\[
P=\begin{pmatrix}a & b\\
b & -2
\end{pmatrix}
\]
for some $a,b\in\mathbb{Z}$. By Theorem \ref{thm:O-Sz-QHS3}, there
must exist isomorphisms 
\[
\Gamma_{P}\xrightarrow[\cong]{\phi}\text{Spin}^{c}(\pm Y)\cong H^{2}(Y;\mathbb{Z})\xrightarrow[\cong]{\text{P.D.}}H_{1}(Y;\mathbb{Z})
\]
where the isomorphism labeled ``P.D.'' is from Poincar\'{e} duality
and the order of $H_{1}(Y;\mathbb{Z})$ is equal to $\det K$. The
matrix $P$ presents the group $\mathbb{Z}/(\det P)\mathbb{Z}$. Therefore,
we must have $\det P=\pm\det K$. Since $\det K$ is odd, we have
that 
\[
b^{2}\equiv-2a-b^{2}=\det P\equiv\det K\equiv1\mod2
\]
and hence $b$ is odd. Therefore, we can use simultaneous row and
column operations to change $P$ into a matrix of form
\[
Q=\begin{pmatrix}a & 1\\
1 & -2
\end{pmatrix}.
\]
Since $Q$ is negative definite, $\det Q\geq0$, so that we must have
$\det Q=+\det K$. Therefore, $a=-(\det K+1)/2$. It follows from
Theorem \ref{thm:Montesinos-theorem} that $m_{Q}(g)\leq m_{G}(g)$
and that the two are congruent modulo $2$. The corollary follows.\end{proof}
\begin{note}
Ozsv\'{a}th and Szab\'{o} used a similar process to obstruct knots from having
unknotting number $1$ in \cite{ozsvath_knots_2005}, although their
isomorphisms $\phi$ were also required to satisfy a ``symmetry''
condition which is not necessarily satisfied in our case. In \cite[Corollary 1.3]{ozsvath_knots_2005},
Ozsv\'{a}th and Szab\'{o} computed the $m_{Q}$ and $m_{G}$ for various knots
to determine whether there exist isomorphisms $\phi$ of the type
given in Corollary \ref{cor:Owens-obstruction-for-alternating-knots}.
The only knot with signature $0$ which had its unknotting number
determined by Ozsv\'{a}th and Szab\'{o} for which the untwisting number was
unknown and for which the ``symmetry'' condition was not necessary
is $10_{68}$. In this way, we get from their computations that $tu(10_{68})=2=u(10_{68})$,
even though $u_{a}(10_{68})=1$. 
\end{note}

\section{\label{sec:Examples}Examples}

In this section, we will prove Theorems \ref{thm:10-crossing-untwisting-numbers}
and \ref{thm:tu-for-11-crossing-knots} using Corollary \ref{cor:Owens-obstruction-for-alternating-knots}.
Following Ozsv\'{a}th and Szab\'{o} in \cite{ozsvath_knots_2005}, we will
refer to an isomorphism $\phi$ satisfying (\ref{eq:positive-matching})
as a \emph{positive matching }and an isomorphism $\phi$ satisfying
(\ref{eq:even-matching}) as an \emph{even matching}. We obstruct
the existence of positive, even matchings for each of the cases listed
in Theorem \ref{thm:10-crossing-untwisting-numbers}. We illustrate
the proof that $tu(10_{68})=2$; the remaining knots are obstructed
from having untwisting number $+1$ and/or $-1$ similarly.
\begin{example}
Although Ozsv\'{a}th and Szab\'{o} have already verified in \cite{ozsvath_knots_2005}
that $\Sigma(10_{68})$ cannot bound a $4$-manifold with intersection
form
\[
Q=\begin{pmatrix}-29 & 1\\
1 & -2
\end{pmatrix},
\]
as it would have to if $tu(10_{68})=1$, we replicate the computation
below. The knot $10_{68}$ has $\sigma(10_{68})=0$, $\det10_{68}=57$,
and Goeritz matrix
\[
G=\begin{pmatrix}-4 & 1 & 1 & 0 & 0 & 1 & 0\\
1 & -3 & 0 & 0 & 1 & 0 & 0\\
1 & 0 & -2 & 1 & 0 & 0 & 0\\
0 & 0 & 1 & -2 & 1 & 0 & 0\\
0 & 1 & 0 & 1 & -3 & 0 & 1\\
1 & 0 & 0 & 0 & 0 & -2 & 1\\
0 & 0 & 0 & 0 & 1 & 1 & -2
\end{pmatrix}.
\]
 The values of $m_{G}$ mod $2$ are 

\begin{flushleft}
	\mylist
0 98/57 50/57 28/19 86/57 56/57 36/19 14/57 2/57 24/19 110/57 2/57 30/19 32/57 56/57 16/19 8/57 50/57 20/19 2/3 98/57 4/19 8/57 86/57 6/19 32/57 14/57 26/19 110/57 110/57 26/19 14/57 32/57 6/19 86/57 8/57 4/19 98/57 2/3 20/19 50/57 8/57 16/19 56/57 32/57 30/19 2/57 110/57 24/19 2/57 14/57 36/19 56/57 86/57 28/19 50/57 98/57. !
\end{flushleft}If $\Sigma(10_{68})$ bounded a $4$-manifold $W$ as in Corollary
\ref{cor:Owens-obstruction-for-alternating-knots}, the matrix 
\[
Q=\begin{pmatrix}a & 1\\
1 & -2
\end{pmatrix}
\]
representing the intersection form on $W$ would have determinant
$-2a-1=\det(10_{68})=57$, so that $a=-29$ and
\[
Q=\begin{pmatrix}-29 & 1\\
1 & -2
\end{pmatrix}.
\]
In this case, the values of $m_{Q}$ mod $2$ are

\begin{flushleft}
	\mylist
0 112/57 106/57 32/19 82/57 64/57 14/19 16/57 100/57 22/19 28/57 100/57 18/19 4/57 64/57 2/19 58/57 106/57 12/19 4/3 112/57 10/19 58/57 82/57 34/19 4/57 16/57 8/19 28/57 28/57 8/19 16/57 4/57 34/19 82/57 58/57 10/19 112/57 4/3 12/19 106/57 58/57 2/19 64/57 4/57 18/19 100/57 28/57 22/19 100/57 16/57 14/19 64/57 82/57 32/19 106/57 112/57. !
\end{flushleft}These lists are not identical (in particular, there is a $112/57$
in the $m_{Q}$ list but not in the $m_{G}$ list), so there are no
even matchings here and $tu(10_{68})\neq+1$. 

The Goeritz matrix for $-10_{68}$ is
\[
G^{\prime}=\begin{pmatrix}-3 & 1 & 0\\
1 & -5 & 3\\
0 & 3 & -6
\end{pmatrix};
\]
the values of $m_{G^{\prime}}$ are 

\begin{flushleft}
	\mylist
0 4/57 16/57 12/19 -50/57 -14/57 10/19 -32/57 28/57 -6/19 -56/57 28/57 2/19 -8/57 -14/57 -4/19 -2/57 16/57 -24/19 -2/3 4/57 -20/19 -2/57 -50/57 -30/19 -8/57 -32/57 -16/19 -56/57 -56/57 -16/19 -32/57 -8/57 -30/19 -50/57 -2/57 -20/19 4/57 -2/3 -24/19 16/57 -2/57 -4/19 -14/57 -8/57 2/19 28/57 -56/57 -6/19 28/57 -32/57 10/19 -14/57 -50/57 12/19 16/57 4/57. !
\end{flushleft}

Using a Python program, we check all possible isomorphisms $\phi$
and find that there are no positive, even matchings between the values
of $m_{Q}$ and the values of $m_{G^{\prime}}$. Therefore, $tu(10_{68})\neq-1$.
Since $u(10_{68})=2$, we must have that $tu(10_{68})=2$ as well.
\end{example}

\section{\label{sec:obstructions-to-tu_p}Ozsv\'{a}th-Szab\'{o} $\tau$ invariant
and Rasmussen $s$ invariant obstructions to $p$-untwisting number}

In this section, we investigate $p$-generalized crossing changes
for fixed $p$ in order to prove Theorem \ref{thm:tu_p-obstruction}. 

Every $p$-generalized crossing change consists of $p(p-1)+p^{2}=p(2p-1)$
standard crossing changes. Thus, for every positive integer
$p$ and every knot $K\subset S^{3}$, if $tu_{p}(K)\leq n$,
then there is an unknotting sequence consisting of $pn(2p-1)$ crossing
changes, so that
\[
u(K)\leq p(2p-1)tu_{p}(K)
\]
whence 
\[
|\tau(K)|\leq u(K)\leq p(2p-1)tu_{p}(K).
\]

Thus, it is possible to use the $\tau$ invariant to get lower bounds
on $tu_{p}$ for all $p$. These bounds may be useful in distinguishing
$tu_{p}$ from $tu_{q}$ for $p\neq q$. However, we may obtain a
stronger bound using the smooth $4$-genus as follows. While visiting Mark Powell at the Max Planck Institute, he suggested this theorem and told me an outline of a somewhat more complicated proof. It is similar to the proof of Powell and coauthors T. Cochran, S. Harvey, and A. Ray that the $\tau$ and $s$ invariants give lower bounds for their bipolar metrics (to appear in a future paper). The following, simpler proof involving the $4$-genus was suggested by the referee.

\begin{thm}
\label{thm:g_4-and-gccs} If $K$ can be unknotted by $k$ generalized crossing changes, where the $i$th change is performed on $2q_{i}$ strands, then

\[
g_{4}(K)\leq \sum_{i=1}^{k} q_{i}^{2}.
\]
\end{thm}

\begin{proof} Suppose that $K$ may be unknotted via $k$ generalized crossing
changes. Then there is a sequence of $k$ generalized crossing changes
taking $K$ to $U$:
\[
K=K_{0}\xrightarrow{q_{1}\text{-gcc}}K_{1}\xrightarrow{q_{2}\text{-gcc}}\dots\xrightarrow{q_{k-1}\text{-gcc}}K_{k-1}\xrightarrow{q_{k}\text{-gcc}}K_{k}=U.
\]
for which $K_{i}$ is obtained from $K_{i-1}$ by a single $q_{i}$-generalized crossing change for $i=1,\dots,k$. Let $D_{i}$ be the disk bounded by the unknot $U_{i}$ on which the $i$th $q_{i}$-generalized crossing change is performed.

First, note that we can isotope $D_{i}$ so that the strands of $K_{i-1}$
pass through it as in Figure \ref{fig:isotoped-U_i}. The strands
passing through $D_{i}$ are oriented in two different ways; we separate
the $q_{i}$ strands of each orientation as in the figure. Let us arbitrarily
call one group of $q_{i}$ strands (say, the ones on the top of the figure)
``left-oriented'' and the other group ``right-oriented''. Hence,
we may assume without loss of generality that we have a local picture
as in Figure \ref{fig:isotoped-U_i}.

A $q_{i}$-generalized crossing change can be undone by changing $q_{i}(2q_{i}-1)$ crossings; one changes precisely one crossing between the $i$th and $j$th strands ($s_{i}$ and $s_{j}$) for each $1\leq i<j\leq 2q_{i}$. Since $q_{i}$ of the strands are oriented in one direction and $q_{i}$ in the other, $q_{i}^{2}$ of these crossing changes occur between strands oriented in opposite directions and $q_{i}(q_{i}-1)$ occur between strands oriented in the same direction (see Figure \ref{fig:Positive-gcc} for an illustration in the case of a $4$-generalized crossing change). Thus, $q_{i}^{2}$ of the crossing changes have one sign, and $q_{i}^{2}-q_{i}$ have the other sign.
Therefore, $K$ can be unknotted by changing $P$ positive crossings and $N$ negative crossings, where
\[
\max\{P,N\}\leq \sum_{i=1}^{k} q_{i}^{2}.
\]
However, it is well known, for instance by the argument in the third paragraph of the introduction of \cite{owens_slicing_2010}, that if $K$ can be unknotted by changing $P$ positive crossings and $N$ negative crossings, then $g_{4}(K)\leq \max\{P,N\}$.
\end{proof}

\begin{figure}
\def\svgwidth{0.2\columnwidth}
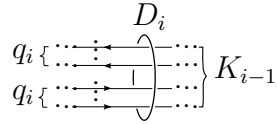

\protect\caption{\label{fig:isotoped-U_i}The result of the isotopy on $D_{i}$ and the strands of $K_{i-1}$. We call the strands on the top \emph{left-oriented}
and those on the bottom \emph{right-oriented}\@.}
\end{figure}

\begin{figure}
\def\svgwidth{0.7\columnwidth}
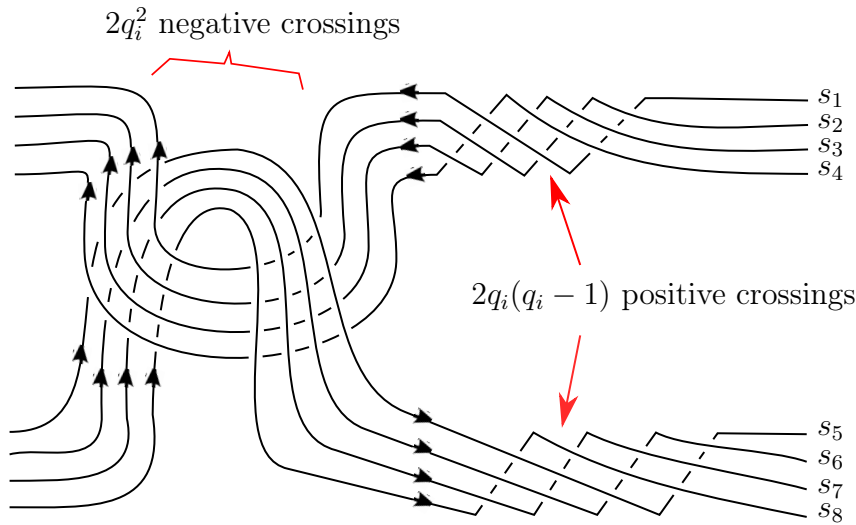

\protect\caption{\label{fig:Positive-gcc}Two sets of four strands twisted around each
other at a positive $4$-generalized crossing change.}
\end{figure}

Since the Ozsv\'{a}th-Szab\'{o} $\tau$ invariant and Rasmussen $s$ invariant give lower bounds on the slice genus of any knot, we immediately get the following:

\begin{cor}
\label{thm:tau-and-s-for-tu_p}Let $K$ be a knot which can be converted
to the unknot via $k$ generalized crossing changes, where the $i$th
generalized crossing change is performed on $2q_{i}$ strands for
$i=1,\dots,k$. Then 
\[
|\tau(K)|\leq\sum_{i=1}^{k}q_{i}^{2}
\]
and
\[
\frac{|s(K)|}{2}\leq\sum_{i=1}^{k}q_{i}^{2}.
\]

\end{cor}

This corollary gives rise to a method for distinguishing $tu_{q}(K)$
from $tu_{p}(K)$ for some $p,q>1$. Suppose that $tu_{q}(K)\leq n$.
Then there exists an untwisting sequence for $K$ consisting of $n$
generalized crossing changes on $2p_{i}$ strands each, where $i=1,\dots,n$
and $p_{i}\leq q$ for all $i$. Applying the corollary, we get that
\[
|\tau(K)|\leq\sum_{i=1}^{n}p_{i}^{2}\leq\sum_{i=1}^{n}q^{2}=nq^{2}
\]
so that we must have 
\[
n\geq\frac{|\tau(K)|}{q^{2}}
\]
and similarly for $|s(K)|/2$ in place of $|\tau(K)|$. We thus obtain
the following obstruction to $tu_{q}(K)=n$:
\begin{cor}
For all integers $q\geq 1$ and all knots $K \subset S^{3}$,
\[
tu_{q}(K)\geq \frac{|\tau(K)|}{q^{2}}\text{ and }tu_{q}(K)\geq \frac{|s(K)|}{2q^{2}}.
\]
\end{cor}
\begin{note}
The above obstruction shows that $|\tau(K)|\leq p^{2}\cdot tu_{p}(K)$
for all $K$, which is stronger than the obstruction $|\tau(K)|\leq p(2p-1)tu_{p}(K)$
given by representing a $p$-generalized crossing change as $p(2p-1)$
standard crossing changes.
\end{note}
\begin{example}
Let $K_{p^{3}}$ denote the ($p^{3},1$)-cable of a knot $K$ with $u(K)=1=\tau(K)=g(K)$
(one example is the right-handed trefoil knot). We know from \cite[Section 5.1]{ince_untwisting_2015}
that $tu_{p^{3}}(K_{p^{3}})=1$ and that $\tau(K_{p^{3}})=p^{3}$. However,
the above result shows that 
\[
tu_{p}(K_{p^{3}})\geq\frac{|\tau(K_{p^{3}})|}{p^{2}}=p.
\]
for all integers $p\geq 1$. Hence
\[
tu_{p}(K_{p^{3}})-tu_{p^{3}}(K_{p^{3}})=tu_{p}(K_{p^{3}})-1\geq p-1\xrightarrow{p\to\infty}\infty.
\]
\end{example}

\bibliographystyle{amsalpha}
\addcontentsline{toc}{section}{\refname}\bibliography{untwisting_information_from_heegaard_floer_homology}

\providecommand{\bysame}{\leavevmode\hbox to3em{\hrulefill}\thinspace}
\providecommand{\MR}{\relax\ifhmode\unskip\space\fi MR }
\providecommand{\MRhref}[2]{%
  \href{http://www.ams.org/mathscinet-getitem?mr=#1}{#2}
}
\providecommand{\href}[2]{#2}
\begin{thebibliography}{{\relax Owe}S13}

\bibitem[Ada94]{adams_knot_1994}
Colin~C. Adams, \emph{The {Knot} {Book}}, American Mathematical Soc., 1994.

\bibitem[CG78]{casson_slice_1978}
Andrew~J. Casson and Cameron~McA Gordon, \emph{On slice knots in dimension
  three}, Proceedings of Symposia in Pure Mathematics \textbf{32} (1978).

\bibitem[CL86]{cochran_unknotting_1986}
Tim~D. Cochran and W.~B.~Raymond Lickorish, \emph{Unknotting information from
  4-manifolds}, Transactions of the American Mathematical Society \textbf{297}
  (1986), no.~1, 125--142.

\bibitem[Inc16]{ince_untwisting_2015}
Kenan Ince, \emph{The untwisting number of a knot}, Pacific Journal of
  Mathematics (accepted) (2016), arXiv: math/1507.04386.

\bibitem[KM93]{kronheimer_gauge_1993}
Peter~B. Kronheimer and Tomasz~S. Mrowka, \emph{Gauge theory for embedded
  surfaces, {I}}, Topology \textbf{32} (1993), no.~4, 773--826.

\bibitem[KM95]{kronheimer_gauge_1995}
\bysame, \emph{Gauge theory for embedded surfaces, {II}}, Topology \textbf{34}
  (1995), no.~1, 37--97.

\bibitem[KT76]{kauffman_signature_1976}
Louis~H. Kauffman and Laurence~R. Taylor, \emph{Signature of links},
  Transactions of the American Mathematical Society \textbf{216} (1976),
  351--365.

\bibitem[Liv02]{livingston_slicing_2002}
Charles Livingston, \emph{The slicing number of a knot}, Algebr. Geom. Topol
  \textbf{2} (2002), 1051--1060.

\bibitem[McC13]{mccoy_alternating_2013}
Duncan McCoy, \emph{Alternating knots with unknotting number one}, arXiv:
  math/1312.1278.

\bibitem[MD88]{mathieu_chirurgies_1988-1}
Yves Mathieu and Michel Domergue, \emph{Chirurgies de {Dehn} de pente $\pm 1$
  sur certains n\ae{}uds dans les 3-vari{\'e}t{\'e}s}, Mathematische Annalen
  \textbf{280} (1988), no.~3, 501--508.

\bibitem[Mon76]{montesinos_three_1976}
Jos{\'e}~M. Montesinos, \emph{Three manifolds as 3-fold branched covers of
  ${S}^{3}$}, Quart. J. Math. Oxford \textbf{2} (1976), no.~27, 85--94.

\bibitem[Mur90]{murakami_algebraic_1990}
Hitoshi Murakami, \emph{Algebraic unknotting operation}, Proceedings of the
  Second Soviet-Japan Symposium of Topology \textbf{8} (1990), 283--292.

\bibitem[NO15]{nagel_unlinking_2015}
Matthias Nagel and Brendan Owens, \emph{Unlinking information from
  4-manifolds}, Bull. London Math. Soc. (accepted) \textbf{47} (2015), no.~6,
  964--979.

\bibitem[OS03a]{ozsvath_absolutely_2003}
Peter Ozsv\'{a}th and Zolt\'{a}n Szab\'{o}, \emph{Absolutely graded {Floer}
  homologies and intersection forms for four-manifolds with boundary}, Advances
  in Mathematics \textbf{173} (2003), no.~2, 179--261.

\bibitem[OS03b]{ozsvath_floer_2003}
\bysame, \emph{On the {Floer} homology of plumbed three-manifolds}, Geometry \&
  Topology \textbf{7} (2003), no.~1, 185--224.

\bibitem[OS05]{ozsvath_knots_2005}
\bysame, \emph{Knots with unknotting number one and {Heegaard} {Floer}
  homology}, Topology \textbf{44} (2005), no.~4, 705--745.

\bibitem[Owe08]{owens_unknotting_2005}
Brendan Owens, \emph{Unknotting information from {Heegaard} {Floer} homology},
  Advances in Mathematics \textbf{217} (2008), no.~5, 2353--2376.

\bibitem[Owe10]{owens_slicing_2010}
\bysame, \emph{On slicing invariants of knots}, Transactions of the American
  Mathematical Society \textbf{362} (2010), no.~6, 3095--3106.

\bibitem[{\relax Owe}S13]{owens_immersed_2013}
Brendan {\relax Owe}ns and Sa\v{s}o Strle, \emph{Immersed disks, slicing
  numbers and concordance unknotting numbers}, Comm. Anal. Geom. (accepted)
  (2013).

\bibitem[Rol76]{rolfsen_knots_1976}
Dale Rolfsen, \emph{Knots and {Links}}, AMS Chelsea Pub., 1976.

\bibitem[Vir73]{viro_branched_1973}
Oleg~J. Viro, \emph{Branched coverings of manifolds with boundary and link
  invariants, {I}}, Mathematics of the USSR-Izvestiya \textbf{7} (1973), no.~6,
  1239.

\end{thebibliography}

\end{document}